\documentclass[12pt,reqno]{amsart}
\usepackage{amssymb}
\usepackage{}
\usepackage{amsfonts}
\usepackage{mathrsfs}
\usepackage{amsmath,amsthm,amssymb,amsfonts,amscd}
\usepackage{mathrsfs}
\usepackage{bbding}
\usepackage{graphicx,latexsym}
\usepackage{etoolbox}

\makeatletter
\patchcmd{\@settitle}{\uppercasenonmath\@title}{}{}{}
\patchcmd{\@setauthors}{\MakeUppercase}{}{}{}
\makeatother

\theoremstyle{plain}
\newtheorem{theorem}{Theorem}[section]

\newtheorem{lemma}{Lemma}[section]
\newtheorem{proposition}[lemma]{Proposition}

\theoremstyle{definition}

\newcommand{\N}{\mathbb N}

\newcommand{\E}{\mathbb E}
\newcommand{\Prob}{\mathbb P}
\newcommand{\ind}{\mathbf 1}
\DeclareMathOperator{\area}{area}

\begin{document}
\medskip

\title[{A converse to the Erd\H{o}s-Fuchs theorem}]{A converse to the Erd\H{o}s-Fuchs theorem}

\author{Quan-Hui Yang}
\address[Quan-Hui Yang]{Ministry of Education Key Laboratory for NSLSCS, School of Mathematical Sciences, Nanjing Normal University, Nanjing 210023, China}
\email{yangquanhui01@163.com}
\author{Lilu Zhao}
\address[Lilu Zhao]
{School of Mathematical Sciences, University of Science and Technology of China, Hefei 230026, China}
\email{zhaolilu@ustc.edu.cn}

\begin{abstract}We prove that there exists $A\subseteq \N$ such that
\[
 R_A(N)=\frac{\pi}{4}N+
 O\!\left(N^{1/4}\sqrt{\log N}\right),
\]
where
$R_A(N)=\#\{(a,b)\in A^2:a+b\le N\}$. 
This improves the record
$O(N^{1/4}\log N)$ obtained by Ruzsa in 1997.
\end{abstract}

\maketitle

\section{Introduction}
\label{sec:intro}

\setcounter{equation}{0}

For an infinite set $A\subseteq\N:=\{0,1,2,\ldots\}$, define
\begin{equation*}
 r_A(n)=\#\{(a,b)\in A\times A:a+b=n\}
\end{equation*}and
\begin{equation}\label{eq:RN}
 R_A(N)=\sum_{n=0}^{N}r_A(n)
       =\#\{(a,b)\in A\times A:a+b\le N\}.
\end{equation}

Erd\H{o}s and Fuchs \cite{EF1956} proved that, for every infinite $A$ and
every $c>0$, the relation
\begin{equation}\label{eq:EF}
 R_A(N)=cN+o\!\left(N^{1/4}(\log N)^{-1/2}\right)
\end{equation}
is impossible. The logarithmic factor in \eqref{eq:EF} was removed by Jurkat (unpublished) and
Montgomery--Vaughan \cite{MV1990} independently. Therefore,
\begin{equation}\label{J-MV}
 R_A(N)=cN+o(N^{1/4})\ \textrm{is impossible. }
\end{equation}

In the opposite direction, Ruzsa \cite{Ruzsa1997} established the existence of an
infinite set $A\subseteq \N$ satisfying
\begin{equation}\label{Ruzsa}
 R_A(N)=\frac{\pi}{4}N+O(N^{1/4}\log N).
\end{equation}
The purpose of this paper is to improve the above result of Ruzsa.

\begin{theorem}\label{thm:main}There exists an infinite subset $A$ in $\N$ such that
\begin{equation}\label{eq:main}
 R_A(N)=\frac{\pi}{4}N+
 O\!\left(N^{1/4}\sqrt{\log N}\right).
\end{equation}
\end{theorem}
Following Ruzsa \cite{Ruzsa1997}, we employ the probabilistic method to establish the existence of the desired set $A$. For each natural number $i$, Ruzsa defined a random variable uniformly distributed in the interval $(i,i+1)$. The new ingredient in this paper lies in introducing one random variable on every interval $(2i,2i+2)$, and this saves a factor $\sqrt{\log N}$ in the bounded differences estimate (see Lemma \ref{lem:square-sum}). 

The conclusion \eqref{J-MV} of Jurkat and
Montgomery--Vaughan is equivalent to
\[
 \limsup_{N\to\infty}\frac{|R_A(N)-cN|}{N^{1/4}}>0.
\]
The order $N^{1/4}$ in the above is still best to date. However, Chen and Tang \cite{CT2011} obtained an explicit lower bound 
\begin{equation}\label{eq:CTconstant}
 \limsup_{N\to\infty}
 \frac{|R_A(N)-cN|}{(cN)^{1/4}}
 \ge \frac{4}{9(25\pi)^{1/4}} .
\end{equation}
For the generalization of Ruzsa's \eqref{Ruzsa}, one may refer to the work of Dai and Pan \cite{DP2014}.

\section{Preliminaries}

\setcounter{equation}{0}

We use the following bounded differences
inequality of McDiarmid \cite{McD1989}.
\begin{lemma}[Bounded differences inequality]\label{lem:BD}
Let $V_1,\ldots,V_m$ be independent random variables, and let
$Z=F(V_1,\ldots,V_m)$ be a bounded measurable real-valued function.
Suppose that changing only the $j$th coordinate can change $F$ by at
most $c_j$, uniformly over the other coordinates. Set
$B=\sum_{j=1}^{m}c_j^2$. If $B>0$, then
\begin{equation*}
 \Prob(|Z-\E Z|>t)\le2\exp(-2t^2/B)\qquad(t>0).
\end{equation*}
\end{lemma}

Let $U_0,U_1,\ldots$ be independent random variables, each uniformly
distributed on $(0,1)$, and set
\begin{equation*}
 x_{2k}=2k+U_k,\qquad x_{2k+1}=2k+2-U_k.
\end{equation*}
Each $x_j$ is uniform in $(j,j+1)$. We say $x_{2k},x_{2k+1}$ is a pair, and different pairs are independent. A pair satisfies
$$x_{2k}+x_{2k+1}=4k+2.$$

For $N\ge0$, write
\begin{equation}\label{eq:SN}
 S(N)=\#\{(i,j)\in\N^2:x_i^2+x_j^2\le N\}.
\end{equation}
For fixed $N$, only finitely many indices contribute to $S(N)$.
Indeed, $x_i^2+x_j^2\le N$ implies $i<\sqrt N$ and $j<\sqrt N$. 
All expectation calculations below therefore involve finite sums.
We first establish the expectation of $S(N)$.
\begin{lemma}\label{lem:mean}
For every $N\ge0$, one has
\begin{equation}\label{eq:mean}
 \left|\E S(N)-\frac{\pi}{4}N\right|\le4.
\end{equation}
\end{lemma}

\begin{proof}
Let
\[
 D_N=\{(x,y)\in[0,\infty)^2:x^2+y^2\le N\}
\]
and
\[
 d_{ij}=\area\bigl(D_N\cap([i,i+1]\times[j,j+1])\bigr).
\]
The area of the boundary is
zero, and hence
\begin{equation*}
 \sum_{i,j\ge0}d_{ij}=\area(D_N).
\end{equation*}
If $x_i,x_j$ lie in different pairs, then the independence gives
\begin{equation*}
 \E\ind_{\{x_i^2+x_j^2\le N\}}=d_{ij}.
\end{equation*}

If $x_i,x_j$ lie in the same pair, then $x_i,x_j\in Q_k$ for some $k$, where
\[
 Q_k=[2k,2k+2]\times[2k,2k+2].
\]
We define
\begin{equation}\label{eq:defineTk}
 T_k(N)=\sum_{i,j\in\{2k,2k+1\}}
       \ind_{\{x_i^2+x_j^2\le N\}}
       \end{equation}
       and its corresponding area contribution
\[
 v_k(N)=\area(D_N\cap Q_k)=\sum_{i,j\in\{2k,2k+1\}}d_{i,j}.
\]
Both $\E T_k(N)$ and $v_k(N)$ belong to the interval $[0,4]$.

If $Q_k$ is entirely outside $D_N$, that is $N\le8k^2$, then $\E T_k(N)=v_k(N)=0$. If $Q_k$ is entirely inside $D_N$, that is $N\ge8(k+1)^2$, then $\E T_k(N)=v_k(N)=4$.
In both two cases we have $\E T_k(N)=v_k(N)$. The last is to consider the possible case that $Q_k$ is partially intersected with $D_N$, and in this case we have
\begin{equation*}
 8k^2<N<8(k+1)^2.
\end{equation*}
There is at most one such $k$, and thus
\begin{equation*}
 \left|\E S(N)-D_N\right|\le 4.
\end{equation*}
We complete the proof by noting that $\area(D_N)=\frac{\pi}{4}N$.
\end{proof}

\section{Bounded differences estimates}

\setcounter{equation}{0}

Recall that
\begin{equation*}
 j<x_j<j+1\qquad(j\ge0).
\end{equation*}

\begin{lemma}\label{lem:one-per-interval}Define
\[
 C_k(y)=\#\{j\notin\{2k,2k+1\}:x_j\le y\}.
\]
One has
\begin{equation}\label{eq:Ck}
 |C_k(y)-y|\le3\qquad(y\ge0).
\end{equation}
\end{lemma}

\begin{proof}
Let $m=\lfloor y\rfloor$ and $\mathcal{Y}=\{j:x_j\le y\}$. Then $j\in \mathcal{Y}$ for $j\le m-1$ and $j\not\in \mathcal{Y}$ for $j\ge m+1$. Thus, $\#\mathcal{Y}\in \{m,m+1\}$. Note that $0\le \#\mathcal{Y}-C_k(y)\le 2$. This proves \eqref{eq:Ck}.
\end{proof}

We define
\begin{equation*}
 f_N(x)=\sqrt{(N-x^2)_+}\qquad(x\ge0),
\end{equation*}
where $(N-x^2)_+=\max(N-x^2,0)$. Then define
\begin{equation}\label{eq:Gk}
 G_{k,N}(u)=f_N(2k+u)+f_N(2k+2-u),\qquad0<u<1.
\end{equation}
For a bounded real function $G$ on $(0,1)$, write
$$\Delta_G=\sup_{0<u<1}G(u)-\inf_{0<u<1}G(u).$$

\begin{lemma}\label{lem:coordinatechange}
Fix all values of $U_l$ with $l\ne k$. Then changing $U_k$ can change
$S(N)$ by at most
\begin{equation}\label{eq:coordinatebound}
 28+2\Delta_{G_{k,N}}.
\end{equation}
If $2k\ge \sqrt{N}$, then changing $U_k$ does change
$S(N)$.
\end{lemma}

\begin{proof}On writing $x_{2k}=2k+u$ and $x_{2k+1}=2k+2-u$, the ordered index pair $i,j$ having exactly one
index in $\{2k,2k+1\}$ contribute
\begin{equation}\label{eq:Ckappear}
 2C_k\bigl(f_N(2k+u)\bigr)
 +2C_k\bigl(f_N(2k+2-u)\bigr)
\end{equation}
to $S(N)$. The factor $2$ above is due to the symmetry. Thus, we have
$$S(N)=2C_k\bigl(f_N(2k+u)\bigr)
 +2C_k\bigl(f_N(2k+2-u)\bigr)+T_{k}(N)+\textrm{OT},$$
 where $T_k(N)$ is given in \eqref{eq:defineTk} and  OT is the contribution from $i,j\not\in \{2k,2k+1\}$. In particular, OT is independent of $u$.

By \eqref{eq:Ck} and \eqref{eq:Gk}, we obtain that \eqref{eq:Ckappear} has the form
\begin{equation*}
 2G_{k,N}(u)+\varepsilon_k(u),\qquad
 |\varepsilon_k(u)|\le12.
\end{equation*}
Recall that $T_k(N)\in[0,4]$. Then changing $U_k$ can change
$S(N)$ by at most
\[
 2|G_{k,N}(u)-G_{k,N}(v)|
 +|\varepsilon_k(u)-\varepsilon_k(v)|+4
 \le2\Delta_{G_{k,N}}+24+4.
\]
This proves \eqref{eq:coordinatebound}.

If $2k\ge \sqrt{N}$, then $S(N)$ is independent of $U_k$. Thus, changing $U_k$ does change
$S(N)$. The proof of this lemma is complete.
\end{proof}

\begin{lemma}\label{lem:oscillation}
If $2k+2\le \sqrt{N}-1$ and set
$\delta_k:=\sqrt{N}-(2k+2)$, then
\begin{equation}\label{eq:interior}
 \Delta_{G_{k,N}}\le \frac{N^{1/4}}{\delta_k^{3/2}}.
\end{equation}
If $\sqrt{N}-3<2k<\sqrt{N}$, then one has
\begin{equation}\label{eq:endpoint}
 \Delta_{G_{k,N}}\le6\,N^{1/4}.
\end{equation}
\end{lemma}

\begin{proof}We write 
$$X=\sqrt{N}.$$
First consider $2k+2\le \sqrt{N}-1$. 
For $x\in[2k,2k+2]$, one has $R-x\ge \delta_k$, $R+x\ge R$ and (due to $0\le x<R$)
\begin{equation*}
 f_N'(x)=-\frac{x}{\sqrt{R^2-x^2}},\qquad
 f_N''(x)=-\frac{R^2}{(R^2-x^2)^{3/2}}.
\end{equation*}
Consequently,
\begin{equation}\label{eq:secondderivative}
 \sup_{x\in[2k,2k+2]}|f_N''(x)|
 \le\frac{R^2}{(\delta_kR)^{3/2}}
 =\frac{\sqrt R}{\delta_k^{3/2}}.
\end{equation}
Write $t=1-u$. Then
\[
 G_{k,N}(u)=g(t):=f_R(2k+1-t)+f_R(2k+1+t),\qquad0<t<1.
\]
The derivative
\[
 g'(t)=f_N'(2k+1+t)-f_N'(2k+1-t)
\]
is a difference rather than a sum.
By the fundamental theorem of calculus and
\eqref{eq:secondderivative}, 
$$|g'(t)|\le2t\frac{\sqrt R}{\delta_k^{3/2}},$$
and  hence
\[
 \Delta_{ G_{k,N}}=\Delta_{g}
 \le\int_0^1|g'(t)|\,dt\le \frac{\sqrt R}{\delta_k^{3/2}},
\]
which is \eqref{eq:interior}.

Now assume $R-3<2k<R$. One has the trivial bound $0\le f_N(x)\le R$, $0\le G_{k,N}(u)\le 2R$ and thus $\Delta_{G_{k,N}}\le 2R$. This proves \eqref{eq:endpoint} when $R\le 3$. Next, we assume that $R>3$. Then
\[
 0\le f_N(x)\le\sqrt{R^2-(R-3)^2}\le\sqrt{6R}\le 3\,\sqrt{R}.
\]
Thus $G_{k,N}$ takes values in $[0,3\,\sqrt{R}]$, proving
\eqref{eq:endpoint}.

The proof of the lemma is complete.
\end{proof}
Now we introduce 
\begin{equation*}
 c_k(N)=
 \begin{cases}
  28+2\,N^{1/4}\,\delta_k^{-3/2},
       &2k+2\le \sqrt{N}-1,\\[2pt]
  28+12\,N^{1/4},&\sqrt{N}-3<2k<\sqrt{N},\\[2pt]
  0,&2k\ge \sqrt{N},
 \end{cases}
\end{equation*}
where 
$\delta_k=\sqrt{N}-(2k+2)$. As an input to apply the bounded differences inequality, we have the following.
\begin{lemma}\label{lem:square-sum}For fixed values of all $U_l$ with $l\ne k$, changing $U_k$ can change
$S(N)$ by at most $c_k(N)$. Moreover,
\begin{equation}\label{eq:square-sum}
 \sum_{k\ge0}c_k(N)^2=O(\sqrt{N}),
\end{equation}
where the $O$-constant is absolute.
\end{lemma}

\begin{proof}The first statement follows from
Lemma~\ref{lem:coordinatechange} and Lemma~\ref{lem:oscillation}.

The following elementary inequality
\begin{equation*}
 \sum_{2k+2\le \sqrt{N}-1}\delta_k^{-3}=O(1)
\end{equation*}proves \eqref{eq:square-sum}. This completes the proof.
\end{proof}

\section{Proof of Theorem \ref{thm:main}}

\setcounter{equation}{0}

\begin{lemma}\label{lem:tail}Let $C_0$ be the absolute $O$-constant in \eqref{eq:square-sum}.
For every real $N\ge16$ and every $t>0$,
\begin{equation}\label{eq:area-tail}
 \Prob\!\left(\left|S(N)-\frac{\pi}{4}N\right|>4+t\right)
 \le2\exp\!\left(-\frac{2t^2}{C_0\sqrt N}\right).
\end{equation}
\end{lemma}
\begin{proof} Let 
$$m=\left\lfloor\frac{\sqrt{N}}{2}\right\rfloor.$$
Then $S(N)$ is bounded and depends only
on the independent random variables $U_{0},\ldots,U_{m}$. 
On applying Lemma~\ref{lem:BD} and Lemma~\ref{lem:square-sum}, we obtain 
\begin{equation*}
 \Prob(|S(N)-\E S(N)|>t)
 \le2\exp\!\left(-\frac{2t^2}{C_0\sqrt N}\right).
\end{equation*}
Now we establish \eqref{eq:area-tail} by combining the above and \eqref{eq:mean}.
\end{proof}

\begin{proposition}\label{prop:almost-sure}
With probability one,
\begin{equation}\label{eq:almost-sure-S}
 S(N)=\frac{\pi}{4}N+
 O\!\left(N^{1/4}\sqrt{\log N}\right).
\end{equation}
\end{proposition}

\begin{proof}
For $N$ sufficiently large, substituting
$t=C_0^{1/2}N^{1/4}\sqrt{\log N}$ in \eqref{eq:area-tail}, we obtain
\begin{equation*}
 \Prob\!\left(\left|S(N)-\frac{\pi}{4}N\right|>4+C_0^{1/2}N^{1/4}\sqrt{\log N}\right)
 \le\frac{2}{N^2}.
\end{equation*}
By the Borel--Cantelli lemma (see \cite[Lemma 8.6.1]{AS} or \cite[Lemma 1.2]{TV}), with probability one we have \eqref{eq:almost-sure-S} holds.
\end{proof}

\begin{proof}[Proof of Theorem~\ref{thm:main}] 
Choose a realization for which \eqref{eq:almost-sure-S} holds by
Proposition~\ref{prop:almost-sure}. Let
\begin{equation*}
 a_j=\lfloor x_j^2\rfloor.
\end{equation*}
Since $j<x_j<j+1$, we have
$$j^2\le a_j<(j+1)^2\le a_{j+1}.$$
Define 
$$A=\{a_j: j\in \N\}.$$ 

Recall \eqref{eq:RN} and \eqref{eq:SN}. We next compare $R_A(N)$ and $S(N)$. For every ordered index pair $(i,j)$, one has
\[
 a_i+a_j\le x_i^2+x_j^2<a_i+a_j+2.
\]
If $x_i^2+x_j^2\le N$, then $a_i+a_j\le N$.
If $a_i+a_j\le N$, then $x_i^2+x_j^2<N+2$. Thus we have
\begin{equation*}
 S(N)\le R_A(N)\le S(N+2),
\end{equation*}
and combining with \eqref{eq:almost-sure-S} gives \eqref{eq:main}. This completes the proof of Theorem \ref{thm:main}.
\end{proof}

\medskip

\section*{Acknowledgments}
This work is support by the National Key Research and Development Program of China (Grant No. 2021YFA1000701), National Natural Science Foundation of China  (Grant No. 12371005 and 12471088).

\end{document}